\documentclass[12pt,a4paper,reqno]{amsart} %Fr o m detta
\usepackage{amssymb}
\usepackage{latexsym}
\usepackage{amsmath}
\usepackage{mathrsfs}
\usepackage{amsthm}
\usepackage{calc}                         %Detta behövs för att fixa
\usepackage[english]{babel}
\usepackage{cite}

\newcommand{\scal}[2]{\langle #1,#2\rangle}

\newcommand{\rr}[1]{\mathbf R^{#1}}

\newcommand{\nn}[1]{\mathbf N^{#1}}
\newcommand{\zz}[1]{\mathbf Z^{#1}}

\newcommand{\nm}[2]{\Vert #1\Vert _{#2}}

\newcommand{\op}{\operatorname{Op}}

\newcommand{\sets}[2]{\{ \, #1\, ;\, #2\, \} }
\newcommand{\ep}{\varepsilon}
\newcommand{\fy}{\varphi}

\newcommand{\cdo}{\, \cdot \, }

\newcommand{\vrum}{\vspace{0.1cm}}
\newcommand{\GL}{\mathbf{M}}

\newcommand{\mabfq}{{\boldsymbol q}}
\newcommand{\mabfp}{{\boldsymbol p}}

\newcommand{\maclS}{\mathcal S}

\newcommand{\mascB}{\mathscr B}

\newcommand{\mascD}{\mathscr D}
\newcommand{\mascE}{\mathscr E}
\newcommand{\mascF}{\mathscr F}
\newcommand{\mascP}{\mathscr P}
\newcommand{\mascS}{\mathscr S}

\numberwithin{equation}{section}          %Detta gör att man får
\newtheorem{thm}{Theorem}
\numberwithin{thm}{section}

\newcommand{\rubrik}{}
\newtheorem{prop}[thm]{Proposition}
\newtheorem{cor}[thm]{Corollary}
\newtheorem{lemma}[thm]{Lemma}

\theoremstyle{definition}

\newtheorem{defn}[thm]{Definition}
\newtheorem{example}[thm]{Example}

\theoremstyle{remark}

\newtheorem{rem}[thm]{Remark}              %T o m hit är bara allmän
\author{Joachim Toft}

\address{Department of Mathematics,
Linn{\ae}us University, Sweden}

\email{joachim.toft@lnu.se}

\title{\textbf{Continuity of Gevrey-H{\"o}rmander pseudo-differential
operators on modulation spaces}}

\begin{document}



\begin{abstract}
Let $s\ge 1$, $\omega ,\omega _0\in \mascP _{E,s}^0$, $a\in
\Gamma _{s}^{(\omega _0)}$, and let $\mascB$ be a suitable invariant
quasi-Banach function space, Then we prove that
the pseudo-differential operator $\op (a)$ is continuous from
$M(\omega _0\omega ,\mascB )$ to $M(\omega ,\mascB )$.
\end{abstract}

\maketitle

%%%%%%%%%%%%%%%%%%%%%%%%%%%%%%%%%
\section{Introduction}\label{sec0}
%%%%%%%%%%%%%%%%%%%%%%%%%%%%%%%%%

\par

The main part of the theory of pseudo-differential operators is given in the
framework of classical function and distribution space theory. That is,
the operators acts between topological vector spaces of functions or distributions
which contain $C_0^\infty$, $\mascS$ or $C^\infty$ and are contained in
corresponding duals $\mascD '$, $\mascS '$ or $\mascE '$ (see \cite{Ho1} or
Section \ref{sec1} for notations).

\par

On the other hand, several problems in e.{\,}g. physics, engineering,
partial differential equations, time-frequency analysis and signal processing are
not well-posed in the framework of such spaces. In such situations
the pseudo-differential operators appearing naturally might fail to have symbols in
classical function and distribution spaces. For example,
Euler-Tricomi equation $D^2_{t}f+tD^2_{x}f=0$, useful in the study of
transonic flow, is not well-posed in the classical setting. Similar facts hold true for
the family of Cauchy problems
\begin{equation}\label{Eq:CauchyProbl}
\begin{cases} D^m_t f -x^{m_1}D_x^{m_2}f=0  \qquad (t,x) \in [0,T]\times \mathbf R 
\\[1ex]
f(0,x)=f_0(x) \qquad \qquad \qquad x \in \mathbf R
\\[1ex]
D_t^k f(0,x) =0, \qquad \qquad  \quad k=1,\ldots, m-1.
\end{cases}
\end{equation}
However, by replacing the classical function and distribution spaces with
suitable Gelfand-Shilov or Gevrey spaces and their spaces of
ultra-distributions, the problem \eqref{Eq:CauchyProbl} become well-posed.
(See \cite{Ca1,Po}.)

\par

An other classical example concerns the heat problem
$$
\partial _tf =\Delta _xf ,\qquad f(0,x)=f_0(x),\ t\in \mathbf R,\ x\in \Omega ,
$$
where $\Omega$ is a cuboid. It is well-posed when moving forward in time ($t>0$),
but not well-posed when moving backwards in
time $(t<0)$ within the framework of classical function and distribution spaces.
On the other hand, by \cite[Example 2.16]{ToNa} it follows that the heat problem
is well-posed for suitable Gelfand-Shilov distribution spaces and Gevrey classes
when $t<0$. Furthermore, if $t>0$, then more precise continuity descriptions is deduced
in the framework of such spaces instead of classical function and distribution spaces.

\par

In the paper we consider continuity properties for a class of pseudo-differential
operators introduced in \cite{CaTo} when acting on a broad class of modulation
spaces. The symbols of the pseudo-differential operators are smooth, should
obey strong ultra-regularity of Gevrey or Gelfand-Shilov types, and are allowed
to grow exponentially or subexponentially.

\par

Related questions were considered in the framework of the usual distribution theory
in \cite{To14}, where pseudo-differential operators were considered, with symbols in
$S^{(\omega _0)}$, the set of all smooth $a$ which satisfies
\begin{equation}\label{Eq:SomegaDef}
|\partial ^\alpha a | \le C _\alpha \omega _0 .
\end{equation}
(See \cite{Ho1} and Section \ref{sec1} for notations.) In \cite[Theorem 3.2]{To14}
it was deduced that if $\mascB$ is a translation invariant BF-space,
$\omega$ and $\omega _0$ belong to $\mascP$, i.{\,}e.
moderate and polynomially bounded weights, and $a\in S^{(\omega _0)}$,
then corresponding pseudo-differential operator, $\op (a)$ is continuous
from the modulation space $M(\omega _0\omega ,\mascB )$ to
$M(\omega ,\mascB )$. The obtained result in \cite{To14} can also
be considered as extensions of certain results in the pioneering paper
\cite{Ta} by Tachizawa. For example, for suitable restrictions on
$\omega$, $\omega _0$ and $\mascB$, it follows that \cite[Theorem 3.2]{To14}
covers \cite[Theorem 2.1]{Ta}.

\par

Several classical continuity properties follows from \cite[Theorem 3.2]{To14}. For
example, since $\mascS$ and $\mascS '$ are suitable intersections and
unions, respectively, of modulation spaces at above, it follows that $\op (a)$
is continuous on $\mascS$ and on $\mascS '$ when $a\in S^{(\omega _0)}$
with $\omega _0\in \mascP$.

\par

Some further conditions on the symbols in $S^{(\omega _0)}$ are required if
corresponding pseudo-differential operators should be continuous on
Gelfand-Shilov spaces, because of the imposed Gevrey regularity on the
elements in such spaces. For symbols $a$ in $\Gamma ^{(\omega _0)}_s$
and $\Gamma ^{(\omega _0)}_{0,s}$, the condition
\eqref{Eq:SomegaDef} is replaced by refined Gevrey-type conditions of the form
\begin{equation}\label{Eq:GammaDef}
|\partial ^\alpha a | \le C h^{|\alpha |}\alpha !^s \omega _0,
\end{equation}
involving global constants $C$ and $h$ which are independent of the order of the
derivatives $\alpha$
(cf. \cite{CaTo}). More precisely, $\Gamma ^{(\omega _0)}_s$ consists of all smooth
$a$ such that \eqref{Eq:GammaDef} holds for some constants $C>0$ and $h>0$,
and $a$ belongs to $\Gamma ^{(\omega _0)}_{0,s}$, whenever for every
$h>0$ there is a constant for some $C>0$ (which depends on both $a$ and $h$)
such that \eqref{Eq:GammaDef} holds.
In the case $s\ge 1$, the set $\mascP$ in \cite{To14} of weight functions are 
essentially replaced by the broader classes $\mascP _{E,s}^0$ and $\mascP _{E,s}$ in
\cite{CaTo}. Here $\omega _0 \in \mascP _{E,s}$ whenever $\omega$ is
$v_r$-moderate for some $r>0$, where
\begin{equation}\label{Eq:CondomegaGS}
v_r = e^{r|\cdo |^{\frac 1s}},
\end{equation}
and $\omega _0 \in \mascP _{E,s}^0$ whenever $\omega$ is
$v_r$-moderate for every $r>0$.

\par

%We notice that
%$$
%\mascP \subseteq \mascP _{E,s_1}^0\subseteq \mascP _{E,s_1}
%\subseteq \mascP _{E,s_2}^0, \qquad s_2<s_1.
%$$
%Hence, despite that $\Gamma ^{(\omega _0)}_{0,s}\subseteq
%\Gamma ^{(\omega _0)}_s\subseteq S^{(\omega _0)}$ holds for every
%$\omega _0$,  we have
%%%
%\begin{align*}
%\Gamma ^{(\omega )}_{0,s} &\nsubseteq
%\bigcup _{\omega _0\in \mascP} S^{(\omega _0)}
%\intertext{for some $\omega \in \mascP _{E,s}^0$, and}
%\Gamma ^{(\omega )}_{0,s} &\nsubseteq
%\bigcup _{\omega _0\in \mascP _{E,s}^0} \Gamma ^{(\omega _0)}_{s}
%\end{align*}
%%%
%for some $\omega \in \mascP _{E,s}$.

\par

In \cite{CaTo} it is proved that if $\omega _0\in \mascP _{E,s}$
and $a\in \Gamma ^{(\omega _0)}_{0,s}$, then corresponding
pseudo-differential operators $\op (a)$ is continuous on the Gelfand-Shilov space
$\Sigma _s$ of Beurling type, and its distribution space $\Sigma _s'$. If instead
$\omega _0\in \mascP _{E,s}^0$ and $a\in \Gamma ^{(\omega _0)}_{s}$, then $\op (a)$
is continuous on the Gelfand-Shilov space $\maclS _s$ of Roumieu type, and its distribution
space $\maclS _s'$. (Cf. Theorems 4.10 and 4.11 in \cite{CaTo}.)

\par

In Section \ref{sec2} we enlarge this family of continuity results by deducing
continuity properties for such pseudo-differential operators when acting on a broad
family of modulation spaces. More precisely, if
$\omega _0,\omega \in \mascP _{E,s}^0$, $\mascB$ is a suitable invariant quasi-Banach-Function
space (QBF-space), $M(\omega ,\mascB)$  is the modulation space with respect to $\omega$
and $\mascB$, and $a\in \Gamma ^{(\omega )}_{s}$,
then we show that $\op (a)$ is continuous from $M(\omega _0\omega ,\mascB )$ to
$M(\omega ,\mascB )$, and that the same holds true with $\mascP _{E,s}$
and $\Gamma ^{(\omega )}_{0,s}$ in place of $\mascP _{E,s}^0$
and $\Gamma ^{(\omega )}_{s}$ (cf. Theorems \ref{p3.2}, \ref{p3.2B} and
\ref{Thm:OpCont3}, and Corollary \ref{Cor:OpCont3}). In the case when
$\mascB$ is a Banach space, then the restrictions on $\mascB$ are given
in Definition \ref{bfspaces1}, while if $\mascB$ fails to be a Banach space, then
suitable Lebesgue quasi-norm estimates are imposed on the elements in $\mascB$.

\par

Evidently, by replacing $\mascP _{E,s}^0$
and $\Gamma ^{(\omega )}_{s}$ with $\mascP$ and $S^{(\omega )}$, our results in Section
\ref{sec2}, when
$\mascB$ is a Banach space,
take the same form as the main result Theorem 3.2 in \cite{To14}. Some of the results
in Section \ref{sec2} can therefore be considered as analogies of the results in
\cite{To14} in the framework of
ultra-distribution theory. We also remark that using the fact that
Gelfand-Shilov spaces and their distribution spaces are equal to suitable intersections and
unions of modulation spaces, the continuity results for pseudo-differential operators 
in \cite{CaTo} are straight-forward consequences of Theorems  \ref{p3.2} and \ref{p3.2B}.
We also refer to \cite{GL,Te1,Te2,PT2,PT3,To18,To24} and the references therein for more 
facts about pseudo-differential operators in framework of Gelfand-Shilov
and modulation spaces.

\par

In Section \ref{sec3} we present some examples on continuity properties for
pseudo-differential operators under considerations. These continuity
properties are straight-forward consequences of the main results, Theorems \ref{p3.2}
and \ref{p3.2B}, from Section \ref{sec2}. Especially we explain continuity
in the framework of Sobolev type spaces and weighted $L^2$ spaces, with
exponential weights, as well as continuity of such operators on
$\Gamma ^{(\omega )}_{0,s}$ spaces.

\medspace

The (classical) modulation spaces $M^{p,q}$, $p,q \in [1,\infty]$, as
introduced by Feichtinger in \cite{Fe4},
consist of all tempered distributions whose
short-time Fourier transforms (STFT) have finite mixed $L^{p,q}$
norm. It follows that the parameters $p$ and $q$ to some extent
quantify the degrees of asymptotic decay and singularity of
the distributions in $M^{p,q}$. The theory of modulation spaces was
developed further and generalized in
\cite{FG1,FG2,FG4,Gc1}, where Feichtinger and Gr{\"o}chenig
established the theory of coorbit spaces. In particular, the modulation
space $M^{p,q}_{(\omega )}$, where $\omega$ denotes a
weight function on phase (or time-frequency shift) space, appears as
the set of tempered (ultra-) distributions
whose STFT belong to the weighted and mixed Lebesgue space
$L^{p,q}_{(\omega )}$.

\par

\section*{Acknowledgement}
The author is grateful to Ahmed Abdeljawad for careful reading of the manuscript
and giving valuable comments, leading to improvements of the content and the
style.

\par

%%%%%%%%%%%%%%%%%%%%%%%%%%%%%%%%%
\section{Preliminaries}\label{sec1}
%%%%%%%%%%%%%%%%%%%%%%%%%%%%%%%%%

\par

In this section we discuss basic properties for modulation
spaces and other related spaces. The proofs are in many cases omitted
since they can be found in \cite
{Fe2,Fe3,Fe4,FG1,FG2,FG4,GaSa,Gc2,To20}.

\par

\subsection{Weight functions}\label{subsec1.1}

A \emph{weight} or \emph{weight function} on $\rr d$ is a positive function
in $L^\infty _{loc}(\rr d)$. Let $\omega$ and $v$ be weights on $\rr d$.
Then $\omega$ is called \emph{$v$-moderate} or \emph{moderate},
if
\begin{equation}\label{e1.1}
\omega (x_1+x_2)\lesssim \omega (x_1) v(x_2),\quad x_1,x_2\in \rr d .
\end{equation}
Here $f(\theta )\lesssim g(\theta )$ means that $f(\theta )\le cg(\theta)$ for some
constant $c>0$ which is independent of $\theta$ in the domain of $f$ and $g$.
If $v$ can be chosen as polynomial, then $\omega$ is called a weight of
polynomial type. The weight function $v$ is called \emph{submultiplicative} if
it is even and \eqref{e1.1} holds for $\omega =v$.

\par

We let $\mascP _E(\rr d)$ be the set of all moderate weights on
$\rr d$, and $\mascP (\rr d)$ be the subset of $\mascP  _E(\rr d)$
which consists of all polynomially moderate functions on $\rr d$.
%We also let $\mascP _0(\rr d)$ be the set
%of all smooth $\omega \in\mascP (\rr d)$ such that
%$(\partial ^\alpha \omega )/\omega$ is bounded for every $\alpha$.
We also let $\mascP _{E,s}(\rr d)$ ($\mascP _{E,s}^0(\rr d)$) be the set of
all weights $\omega$ in $\rr d$ such that
\begin{equation}\label{Eq:ModWeightProp}
\omega (x_1+x_2)\lesssim \omega (x_1) e^{r|x_2|^{\frac 1s}},\quad x_1,x_2\in \rr d .
\end{equation}
for some $r>0$ (for every $r>0$). We have
\begin{alignat*}{3}
\mascP &\subseteq \mascP _{E,s_1}^0\subseteq \mascP _{E,s_1}\subseteq
\mascP _{E,s_2}^0\subseteq \mascP _E & \quad &\text{when} & s_2&<s_1 
\intertext{and}
\mascP _{E,s} &= \mascP _E & \quad &\text{when} &\quad  s&\le 1 ,
\end{alignat*}
where the last equality follows from the fact that if $\omega \in \mascP _E(\rr d)$
($\omega \in \mascP _E^0(\rr d)$), then
\begin{equation}\label{Eq:ModWeightPropCons}
\omega (x+y)\lesssim \omega (x) e^{r|y|^{\frac 1s}}
\quad \text{and}\quad
e^{-r|x|}\le \omega (x)\lesssim e^{r|x|},\quad
x,y\in \rr d
\end{equation}
hold true for some $r>0$ (for every $r>0$) (cf. \cite{Gc2.5}).

%\par
%
%Note that if $\omega \in\mascP _E(\rr d)$, then $\omega (x)+\omega
%(x)^{-1}\le e^{r|x|}$, for some $r>0$,
%and
%$\omega \in\mascP (\rr d)$, then $\omega (x)+\omega
%(x)^{-1}\le p(x)$, for some polynomial $p$ on $\rr d$.

\par

\subsection{Gelfand-Shilov spaces}\label{subsec1.2}

\par

Let $0<h,s,\sigma \in \mathbf R$ be fixed. Then $\maclS _{s,h}^\sigma (\rr d)$
consists of all $f\in C^\infty (\rr d)$ such that
\begin{equation}\label{gfseminorm}
\nm f{\maclS _{s,h}^\sigma}\equiv \sup \frac {|x^\beta \partial ^\alpha
f(x)|}{h^{|\alpha  + \beta |}\alpha !^s\, \beta !^t}
\end{equation}
is finite. Here the supremum should be taken over all $\alpha ,\beta \in
\mathbf N^d$ and $x\in \rr d$.

\par

Obviously $\maclS _{s,h}^\sigma$ is a Banach space, contained in $\mascS$,
and which increases with $h$, $s$ and $t$ and 
$\maclS _{s,h}^\sigma \hookrightarrow \mathscr S$. Here and
in what follows we use the notation $A\hookrightarrow B$ when the topological
spaces $A$ and $B$ satisfy $A\subseteq B$ with continuous embeddings.
%Furthermore, if $s,\sigma >\frac 12$, then
%$$
%\maclS _{s,h}^\sigma
%\quad \text{and}\quad
%\bigcup _{h>0} \maclS _{s,h}^\sigma
%$ contains all finite linear
%combinations of Hermite functions. Since such linear combinations
%are dense in $\mathscr S$ and in $\maclS _{s,h}^\sigma$, it follows
%that the dual $(\maclS _{s,h}^\sigma )'(\rr d)$ of $\maclS _{s,h}^\sigma (\rr d)$ is
%a Banach space which contains $\mathscr S'(\rr d)$, for such choices
%of $s$ and $t$.

\par

The \emph{Gelfand-Shilov spaces} $\maclS _{t}^s(\rr d)$ and
$\Sigma _{t}^s(\rr d)$ are defined as the inductive and projective 
limits respectively of $\maclS _{s,h}^\sigma (\rr d)$. This implies that
\begin{equation}\label{GSspacecond1}
\maclS _t^{s}(\rr d) = \bigcup _{h>0}\maclS _{s,h}^\sigma (\rr d)
\quad \text{and}\quad \Sigma _t^{s}(\rr d) =\bigcap _{h>0}
\maclS _{s,h}^\sigma (\rr d),
\end{equation}
and that the topology for $\maclS _t^{s}(\rr d)$ is the strongest
possible one such that the inclusion map from $\maclS _{s,h}^\sigma
(\rr d)$ to $\maclS _t^{s}(\rr d)$ is continuous, for every choice 
of $h>0$. The space $\Sigma _s^\sigma (\rr d)$ is a Fr{\'e}chet space
with seminorms $\nm \cdo{\maclS _{s,h}^\sigma}$, $h>0$. Moreover,
$\Sigma _s^\sigma (\rr d)\neq \{ 0\}$, if and only if $s+\sigma \ge 1$ and
$(s,\sigma )\neq (\frac 12,\frac 12)$, and $\maclS _s^\sigma (\rr d)\neq \{ 0\}$,
if and only if $s+\sigma \ge 1$. If $s$ and $\sigma$ are chosen such that
$\Sigma _s^\sigma (\rr d)\neq \{ 0\}$, then $\Sigma _s^\sigma (\rr d)$
is dense in $\mascS (\rr d)$ and in $\maclS _s^\sigma (\rr d)$.
The same is true with $\maclS _s^\sigma (\rr d)$
in place of $\Sigma _s^\sigma (\rr d)$ (cf. \cite{GS}).

\medspace

The \emph{Gelfand-Shilov distribution spaces} $(\maclS _s^\sigma )'(\rr d)$
and $(\Sigma _s^\sigma )'(\rr d)$ are the projective and inductive limit
respectively of $(\maclS _{s,h}^\sigma )'(\rr d)$.  This means that
\begin{equation}\tag*{(\ref{GSspacecond1})$'$}
(\maclS _s^\sigma )'(\rr d) = \bigcap _{h>0}(\maclS _{s,h}^\sigma )'(\rr d)\quad
\text{and}\quad (\Sigma _s^\sigma )'(\rr d) =\bigcup _{h>0}(\maclS _{s,h}^\sigma )'(\rr d).
\end{equation}
We remark that in \cite{GS} it is proved that $(\maclS _s^\sigma )'(\rr d)$
is the dual of $\maclS _s^\sigma (\rr d)$, and $(\Sigma _s^\sigma )'(\rr d)$
is the dual of $\Sigma _s^\sigma (\rr d)$ (also in topological sense). For conveniency we
set
$$
\maclS _s=\maclS _s^s,\quad \maclS _s'=(\maclS _s^s)',\quad
\Sigma _s=\Sigma _s^s
\quad \text{and}\quad
\Sigma _s'=(\Sigma _s^s)'.
$$

\par

For every admissible $s,\sigma >0$ and $\ep >0$ we have
\begin{multline}\label{GSembeddings}
%\begin{alignedat}{2}
\Sigma _s^\sigma  (\rr d) \hookrightarrow 
\maclS _s^\sigma (\rr d) \hookrightarrow  \Sigma _{s+\ep}^{\sigma +\ep}(\rr d)
\hookrightarrow \mascS (\rr d)
\\[1ex]
%\quad \text{and}\quad
\hookrightarrow \mascS '(\rr d)\hookrightarrow
(\Sigma _{s+\ep}^{\sigma +\ep})' (\rr d) \hookrightarrow  (\maclS _s^\sigma )'(\rr d)
\hookrightarrow  (\Sigma _s^\sigma )'(\rr d).
%\end{alignedat}
\end{multline}

\par

From now on we let $\mathscr F$ be the Fourier transform which
takes the form
$$
(\mathscr Ff)(\xi )= \widehat f(\xi ) \equiv (2\pi )^{-\frac d2}\int _{\rr
{d}} f(x)e^{-i\scal  x\xi }\, dx
$$
when $f\in L^1(\rr d)$. Here $\scal \cdo \cdo$ denotes the usual
scalar product on $\rr d$. The map $\mathscr F$ extends 
uniquely to homeomorphisms on $\mathscr S'(\rr d)$,
from $(\maclS _s^\sigma )'(\rr d)$ to $(\maclS _\sigma ^s)'(\rr d)$ and
from $(\Sigma _s^\sigma )'(\rr d)$ to $(\Sigma _\sigma ^s)'(\rr d)$. Furthermore,
$\mascF$ restricts to
homeomorphisms on $\mathscr S(\rr d)$, from
$\maclS _s^\sigma (\rr d)$ to $\maclS _\sigma ^s(\rr d)$ and
from $\Sigma _s^\sigma (\rr d)$ to $\Sigma _\sigma ^s(\rr d)$,
and to a unitary operator on $L^2(\rr d)$. Similar facts hold true
when $s=\sigma$ and the Fourier transform is replaced by a partial
Fourier transform.

\par

Let $\phi \in \maclS _s^\sigma (\rr d)$ be fixed. Then the \emph{short-time
Fourier transform} $V_\phi f$ of $f\in (\maclS _s ^\sigma )'
(\rr d)$ with respect to the \emph{window function} $\phi$ is
the Gelfand-Shilov distribution on $\rr {2d}$, defined by
$$
V_\phi f(x,\xi ) \equiv  (f,\phi (\cdo -x)e^{i\scal \cdo \xi}).
$$
If in addition $f$ is an integrable function, then
%where $(UF)(x,y)=F(y,y-x)$. Here $\mascF _2F$ is the partial Fourier transform
%of $F(x,y)\in \maclS _s'(\rr {2d})$ with respect to the $y$ variable.
%If $f ,\phi \in \maclS _s (\rr d)$, then it follows that
$$
V_\phi f(x,\xi ) = (2\pi )^{-\frac d2}\int f(y)\overline {\phi
(y-x)}e^{-i\scal y\xi}\, dy .
$$

\par

Gelfand-Shilov spaces and their distribution spaces can in convenient
ways be characterized by means of estimates of Fourier and
short-time Fourier transforms (see e.{\,}g. \cite{ChuChuKim,GZ,To18,To22}).
Here some extension of the map $(f,\phi )\mapsto V_\phi f$ are also given,
for example that this map is uniquely extendable to a continuous map
from $(\maclS _s^\sigma )'(\rr d)\times (\maclS _s^\sigma )'(\rr d)$
to $(\maclS _{s,\sigma}^{\sigma ,s} )'(\rr {2d})$ (see also \cite{AbCaTo}
for notations).

\par

\subsection{Modulation spaces}\label{subsec1.3}

\par

We recall that a quasi-norm $\nm {\cdo}{\mascB}$ of order $r \in (0,1]$ on the
vector-space $\mascB$ over $\mathbf C$ is a nonnegative functional on
$\mascB$ which satisfies
\begin{alignat}{2}
 \nm {f+g}{\mascB} &\le 2^{\frac 1r-1}(\nm {f}{\mascB} + \nm {g}{\mascB}), &
\quad f,g &\in \mascB ,
\label{Eq:WeakTriangle1}
\\[1ex]
\nm {\alpha \cdot f}{\mascB} &= |\alpha| \cdot \nm f{\mascB},
& \quad \alpha &\in \mathbf{C},
\quad  f \in \mascB
\notag
\intertext{and}
   \nm f {\mascB} &= 0\quad  \Leftrightarrow \quad f=0. & &
\notag
\end{alignat}

\par

The vector space $\mascB$ is called a quasi-Banach space if it is a complete quasi-normed
space.
If $\mascB$ is a quasi-Banach space with quasi-norm satisfying \eqref{Eq:WeakTriangle1}
then on account of \cite{Aik,Rol} there is an equivalent quasi-norm to $\nm \cdo {\mascB}$
which additionally satisfies
\begin{align}\label{Eq:WeakTriangle2}
\nm {f+g}{\mascB}^r \le \nm {f}{\mascB}^r + \nm {g}{\mascB}^r, 
\quad f,g \in \mascB .
\end{align}
From now on we always assume that the quasi-norm of the quasi-Banach space $\mascB$
is chosen in such way that both \eqref{Eq:WeakTriangle1} and \eqref{Eq:WeakTriangle2}
hold.

\par

Let $\phi \in \Sigma _1(\rr d)\setminus 0$, $p,q\in (0,\infty ]$
and $\omega \in\mascP _E(\rr {2d})$ be fixed. Then the
\emph{modulation space} $M^{p,q}_{(\omega )}(\rr d)$ consists of all
$f\in \Sigma _1'(\rr d)$ such that
\begin{equation}\label{modnorm}
\nm f{M^{p,q}_{(\omega )}} \equiv \Big ( \int \Big ( \int |V_\phi f(x,\xi
)\omega (x,\xi )|^p\, dx\Big )^{q/p} \, d\xi \Big )^{1/q} <\infty
\end{equation}
(with the obvious modifications when $p=\infty$ and/or
$q=\infty$). We set $M^p_{(\omega )}=M^{p,p}_{(\omega )}$, and
if $\omega =1$, then we set $M^{p,q}=M^{p,q}_{(\omega )}$
and $M^{p}=M^{p}_{(\omega )}$.

\par

The following proposition is a consequence of well-known facts
in \cite {Fe4,GaSa,Gc2,To20}. Here and in what follows, we let $p'$
denotes the conjugate exponent of $p$, i.{\,}e.
$$
p'
=
\begin{cases}
\infty & \text{when}\ p\in (0,1]
\\[1ex]
\displaystyle{\frac p{p-1}} & \text{when}\ p\in (1,\infty )
\\[2ex]
1 & \text{when}\ p=\infty \, .
\end{cases}
$$

\par

\begin{prop}\label{p1.4}
Let $p,q,p_j,q_j,r\in (0,\infty ]$ be such that $r\le \min (1,p,q)$,
$j=1,2$, let $\omega
,\omega _1,\omega _2,v\in\mascP _E(\rr {2d})$ be such that $\omega$
is $v$-moderate, $\phi \in M^r_{(v)}(\rr d)\setminus 0$, and let $f\in
\Sigma _1'(\rr d)$. Then the following is true:
\begin{enumerate}
\item $f\in
M^{p,q}_{(\omega )}(\rr d)$ if and only if \eqref{modnorm} holds,
i.{\,}e. $M^{p,q}_{(\omega )}(\rr d)$ is independent of the choice of
$\phi$. Moreover, $M^{p,q}_{(\omega )}$ is a quasi-Banach space under the
quasi-norm in \eqref{modnorm}, and different choices of $\phi$ give rise to
equivalent quasi-norms.

\par

If in addition $p,q\ge 1$, then
$M^{p,q}_{(\omega )}(\rr d)$ is a Banach space with norm \eqref{modnorm};

\vrum

\item if  $p_1\le p_2$,
$q_1\le q_2$ and $\omega _2\lesssim \omega _1$, then
\begin{alignat*}{3}
\Sigma _1(\rr d)&\subseteq &M^{p_1,q_1}_{(\omega _1)}(\rr
d) &\subseteq  & M^{p_2,q_2}_{(\omega _2)}(\rr d)&\subseteq 
\Sigma _1'(\rr d).
%\\[1ex]
%\Sigma _1(\rr d)&\subseteq &W^{p_1,q_1}_{(\omega _1)}(\rr
%d) &\subseteq  & W^{p_2,q_2}_{(\omega _2)}(\rr d)&\subseteq 
%\Sigma _1'(\rr d) .
\end{alignat*}
\end{enumerate}
\end{prop}

\par

We refer to \cite {Fe4,FG1,FG2,FG4,GaSa,Gc2,RSTT,To20}
for more facts about modulation spaces.

\par

\subsection{A broader family of modulation spaces}

\par

As announced in the introduction we consider in Section  \ref{sec2}
mapping properties for pseudo-differential operators when acting on
a broader class of modulation spaces, which are defined by imposing certain
types of translation invariant solid BF-space norms on the short-time
Fourier transforms. (Cf. \cite{Fe4,Fe6,Fe8,FG1,FG2}.)

\par

\begin{defn}\label{bfspaces1}
Let $\mascB \subseteq L^r_{loc}(\rr d)$ be a quasi-Banach
of order $r\in (0,1]$, and let $v \in\mascP _E(\rr d)$.
Then $\mascB$ is called a \emph{translation invariant
Quasi-Banach Function space on $\rr d$} (with respect to $v$), or \emph{invariant
QBF space on $\rr d$}, if there is a constant $C$ such
that the following conditions are fulfilled:
\begin{enumerate}
\item if $x\in \rr d$ and $f\in \mascB$, then $f(\cdo -x)\in
\mascB$, and 
\begin{equation}\label{translmultprop1}
\nm {f(\cdo -x)}{\mascB}\le Cv(x)\nm {f}{\mascB}\text ;
\end{equation}

\vrum

\item if  $f,g\in L^r_{loc}(\rr d)$ satisfy $g\in \mascB$ and $|f|
\le |g|$, then $f\in \mascB$ and
$$
\nm f{\mascB}\le C\nm g{\mascB}\text .
$$
\end{enumerate}
\end{defn}

\par

If $v$ belongs to $\mascP _{E,s}(\rr d)$
($\mascP _{E,s}^0(\rr d)$) , then $\mascB$ in Definition \ref{bfspaces1}
is called an invariant BF-space of Roumieu type (Beurling type) of order $s$.

\par

It follows from (2) in Definition \ref{bfspaces1} that if $f\in
\mascB$ and $h\in L^\infty$, then $f\cdot h\in \mascB$, and
\begin{equation}\label{multprop}
\nm {f\cdot h}{\mascB}\le C\nm f{\mascB}\nm h{L^\infty}.
\end{equation}
If $r=1$, then $\mascB$ in Definition \ref{bfspaces1} is a Banach
space, and the condition (2)  means that a
translation invariant QBF-space is a solid BF-space in the sense of
(A.3) in \cite{Fe6}. 
The space $\mascB$ in Definition \ref{bfspaces1} is called an
\emph{invariant BF-space} (with respect to $v$) if $r=1$, and
Minkowski's inequality holds true, i.{\,}e.
\begin{equation}\label{Eq:MinkIneq}
\nm {f*\fy}{\mascB}\lesssim \nm {f}{\mascB}\nm \fy{L^1_{(v)}},
\qquad f\in \mascB ,\ \fy \in C_0^\infty (\rr d).
\end{equation}
%%
%for some constant $C$ which is independent of
%$f\in \mascB$ and $\fy \in C_0^\infty (\rr d)$.

\par

\begin{example}\label{Lpqbfspaces}
Assume that $p,q\in [1,\infty ]$, and let $L^{p,q}_1(\rr {2d})$ be the
set of all $f\in L^1_{loc}(\rr {2d})$ such that
$$
\nm  f{L^{p,q}_1} \equiv \Big ( \int \Big ( \int |f(x,\xi )|^p\, dx\Big
)^{q/p}\, d\xi \Big )^{1/q}
$$
if finite. Also let $L^{p,q}_2(\rr {2d})$ be the set of all $f\in
L^1_{loc}(\rr {2d})$ such that
$$
\nm  f{L^{p,q}_2} \equiv \Big ( \int \Big ( \int |f(x,\xi )|^q\, d\xi
\Big )^{p/q}\, dx \Big )^{1/p}
$$
is finite. Then it follows that $L^{p,q}_1$ and $L^{p,q}_2$ are
translation invariant BF-spaces with respect to $v=1$.
\end{example}

\par

For translation invariant BF-spaces we make the
following observation.

\par

\begin{prop}\label{p1.4BFA}
Assume that $v\in\mascP _E(\rr {d})$, and that $\mascB$ is an
invariant BF-space with respect to $v$ such that \eqref{Eq:MinkIneq}
holds true. Then the
convolution mapping $(\fy ,f)\mapsto \fy *f$ from $C_0^\infty (\rr
d)\times \mascB$ to $\mascB$ extends uniquely to a continuous
mapping from
$L^1_{(v )}(\rr d)\times \mascB$ to $\mascB$, and \eqref{Eq:MinkIneq}
holds true for any $f\in \mascB$ and $\fy \in L^1_{(v)}(\rr d)$.
\end{prop}

\par

The result is a straight-forward consequence of the fact that $C_0^\infty$
is dense in $L^1_{(v)}$.

\par

Next we consider the extended class of modulation spaces which we are interested
in. 

\par

\begin{defn}\label{bfspaces2}
Assume that $\mascB$ is a translation
invariant QBF-space on $\rr {2d}$, $\omega \in\mascP _E(\rr {2d})$,
and that $\phi \in
\Sigma _1(\rr d)\setminus 0$. Then the set $M(\omega ,\mascB )$
consists of all $f\in \Sigma _1'(\rr d)$ such that
$$
\nm f{M(\omega ,\mascB )}
\equiv \nm {V_\phi f\, \omega }{\mascB}
$$
is finite.
%%% If $\omega =1$, then the notation $M(\mascB )$ is used
%%% instead of $M(\omega ,\mascB )$.
\end{defn}

\par

Obviously, we have
$
M^{p,q}_{(\omega )}(\rr d)=M(\omega ,\mascB )$
when  $\mascB =L^{p,q}_1(\rr {2d})$ (cf. Example \ref{Lpqbfspaces}).
It follows that many properties which are valid for the classical modulation
spaces also hold for the spaces of the form $M(\omega ,\mascB )$.
For example we have the following proposition, which shows that
the definition of $M(\omega ,\mascB )$ is independent of the
choice of $\phi$ when $\mascB$ is a Banach space. We omit the proof
since it follows by similar arguments as
in the proof of Proposition 11.3.2 in \cite{Gc2}.

\par

\begin{prop}\label{p1.4BF}
Let $\mascB$ be an invariant BF-space with
respect to $v_0\in \mascP _E(\rr {2d})$ for $j=1,2$. Also let
$\omega ,v\in\mascP _E(\rr {2d})$ be such that $\omega$ is
$v$-moderate, $M(\omega ,\mascB )$ is the same as in Definition
\ref{bfspaces2}, and let $\phi \in M^1_{(v_0v)}(\rr d)\setminus
0$ and $f\in \Sigma _1'(\rr d)$. Then $f\in M(\omega ,\mascB )$
if and only if $V_\phi f\, \omega \in \mascB$, and
different choices of $\phi$ gives rise to equivalent norms in
$M(\omega ,\mascB )$.
\end{prop}

\par

Finally we recall the following result on completeness for $M(\omega ,\mascB )$.
We refer to \cite{To22} for a proof of the first assertion and \cite{PfTo} for the second
one.

\par

\begin{prop}\label{Prop:ModCompleteness}
Let $\omega$
be a weight on $\rr {2d}$, and let $\mascB $ be an invariant
QBF-space with respect to the submultiplicative
$v\in \mascP _E(\rr {2d})$.
Then the following is true:
\begin{enumerate}
\item if in addition $\mascB$ is a mixed quasi-norm space
of Lebesgue types, then $M(\omega ,\mascB )$ is a quasi-Banach space;

\vrum

\item if in addition $\mascB$ an invariant BF-space with respect to $v$,
then $M(\omega ,\mascB )$ is a quasi-Banach space.
\end{enumerate}
\end{prop}

\par

\subsection{Mixed quasi-normed Lebesgue spaces}
In most cases, the quasi-Banach spaces $\mascB$ are
mixed quasi-normed Lebesgue space, which are defined next. 
Let $E= \{ e_1,\dots,e_d \}$ be an orderd basis of $\rr d$.
Then the corresponding lattice is 
\begin{equation*}
\Lambda _E =\sets{j_1e_1+\cdots +j_de_d}{(j_1,\dots,j_d)\in \zz d},
\end{equation*}

\par

We define for each $\mabfq =(q_1,\dots ,q_d)\in (0,\infty ]^d$
$$
\max \mabfq =\max (q_1,\dots ,q_d)
\quad \text{and}\quad
\min \mabfq =\min (q_1,\dots ,q_d).
$$

\par

\begin{defn}\label{Def:MixedLebSpaces}
Let $E = \{ e_1,\dots ,e_d\}$ be an orderd basis of $\rr d$, $\omega$ be a weight on
$\rr d$,
$\mabfp =(p_1,\dots ,p_d)\in (0,\infty ]^{d}$ and $r=\min (1,\mabfp )$.
If  $f\in L^r_{loc}(\rr d)$, then
$$
\nm f{L^{\mabfp }_{E,(\omega )}}\equiv
\nm {g_{d-1}}{L^{p_{d}}(\mathbf R)},
$$
where  $g_k(z_k)$, $z_k\in \rr {d-k}$,
$k=0,\dots ,d-1$, are inductively defined as
\begin{align*}
g_0(x_1,\dots ,x_{d}) &\equiv |f(x_1e_1+\cdots +x_{d}e_d)
\omega (x_1e_1+\cdots +x_{d}e_d)|,
\\[1ex]
\intertext{and}
g_k(z_k) &\equiv
\nm {g_{k-1}(\cdo ,z_k)}{L^{p_k}(\mathbf R)},
\quad k=1,\dots ,d-1.
\end{align*}
The space $L^{\mabfp }_{E,(\omega )}(\rr d)$ consists
of all $f\in L^r_{loc}(\rr d)$ such that
$\nm f{L^{\mabfp}_{E,(\omega )}}$ is finite, and is called
\emph{$E$-split Lebesgue space (with respect to $\mabfp$ and $\omega$}).
\end{defn}

\par

Let $E$, $\mabfp$ and $\omega$ be the same as in Definition \ref{Def:MixedLebSpaces}.
Then the discrete version $\ell ^{\mabfp}_{E,(\omega )}(\Lambda _E)$
of $L^{\mabfp}_{E,(\omega )}(\rr d)$ is the set of all
sequences $a=\{ a(j)\} _{j\in \Lambda _E}$ such that the quasi-norm
$$
\nm a{\ell ^{\mabfp}_{E,(\omega )}}
\equiv
\nm {f_a}{L^{\mabfp}_{E,(\omega )}},\qquad f_a=\sum _{j\in \Lambda _E} a(j)\chi _j,
$$
is finite. Here $\chi _j$ is the characteristic function of $j+\kappa (E)$,
where $\kappa (E)$ is the parallelepiped spanned by the basis $E$.
We also set $L^{\mabfp}_{E} = L^{\mabfp}_{E,(\omega )}$
and $\ell ^{\mabfp}_{E} = \ell ^{\mabfp}_{E,(\omega )}$
when $\omega =1$.

\par

\begin{defn}\label{Def:MixedPhaseShiftLebSpaces}
Let $E$ be an ordered basis of the phase space $\rr {2d}$. Then
$E$ is called \emph{phase split}
if there is a subset $E_0\subseteq E$ such that the span of $E_0$ equals
$\sets {(x,0)\in \rr {2d}}{x\in \rr d}$, and the span of $E\setminus E_0$
equals $\sets {(0,\xi )\in \rr {2d}}{\xi \in \rr d}$.
\end{defn}

\par

\subsection{Pseudo-differential operators}

\par

Next we recall some facts on pseudo-differential operators. Let
$A\in \GL (d,\mathbf R)$ be fixed and
let $a\in \Sigma _1(\rr {2d})$. Then the pseudo-differential
operator $\op _A(a)$ is the linear and continuous operator on $\Sigma _1(\rr d)$,
defined by the formula
\begin{multline}\label{e0.5}
(\op _A(a)f)(x)
\\[1ex]
=
(2\pi  ) ^{-d}\iint a(x-A(x-y),\xi )f(y)e^{i\scal {x-y}\xi }\,
dyd\xi .
\end{multline}
The definition of $\op _A(a)$ extends to
any $a\in \Sigma _1'(\rr {2d})$, and then $\op _t(a)$ is continuous from
$\Sigma _1(\rr d)$ to $\Sigma _1'(\rr d)$. Moreover, for every fixed
$A\in \GL (d,\mathbf R)$, it follows that there is a one to
one correspondence between such operators, and pseudo-differential
operators of the form $\op _A(a)$. (See e.{\,}g. \cite {Ho1}.)
If $A=2^{-1}I$, where $I\in \GL (d,\mathbf R)$ is the identity matrix, then
$\op _A(a)$ is equal to the Weyl operator $\op ^w(a)$
of $a$. If instead $A=0$, then the standard (Kohn-Nirenberg)
representation $\op (a)$ is obtained.

\par

If $a_1,a_2\in \Sigma _1'(\rr {2d})$ and $A_1,A_2\in
 \GL (d,\mathbf R)$, then
\begin{equation}\label{pseudorelation}
\op _{A_1}(a_1)=\op _{A_2}(a_2) \quad \Leftrightarrow \quad a_2(x,\xi
)=e^{i\scal {(A_1-A_2)D_\xi}{D_x}}a(x,\xi ).
\end{equation}
(Cf. \cite{Ho1}.)

\par

The following special case of \cite[Theorem 3.1]{To24} is important when discussing
continuity of pseudo-differential operators when acting on quasi-Banach modulation
spaces.

\par

\begin{prop}\label{Prop:OpCont}
Let $\omega _1,\omega _2
\in \mathscr P_{E}(\rr {2d})$ and $\omega _0\in \mathscr P_{E}(\rr {2d}\oplus \rr {2d})$
be such that
\begin{equation}\label{Eq:WeightfracCond1}
\frac {\omega _2(x,\xi  )}{\omega _1
(y,\eta )} \lesssim \omega _0( x,\eta ,\xi -\eta ,y-x ).
\end{equation}
Also let $\mabfp \in (0,\infty]^{2d}$,
$E$ be a phase split basis of $\rr {2d}$  and let
$a\in M^{\infty ,1}_{(\omega _0)}(\rr {2d})$. Then
$\op _0(a)$ is continuous from $M(L^{\mabfp ,E},
\omega _1)$ to $M(L^{\mabfp ,E},\omega _2)$.
\end{prop}

\par

In the next section we discuss continuity for pseudo-differential
operators with symbols in the following definition. (See also the introduction.)

\par

\begin{defn}\label{Def:GevSymbols}
Let $\omega _0$ be a weight on $\rr {d}$, and let $s\ge 0$.
\begin{enumerate}
\item The set $\Gamma ^{(\omega _0)}_{s}(\rr d)$ consists of all $a\in C^\infty (\rr d)$
such that
\begin{equation}\label{Eq:GammaDef2}
|\partial ^\alpha f(x) | \lesssim h^{|\alpha |}\alpha !^s \omega _0(x),\qquad \alpha \in \nn d,
\end{equation}
for some constant $h>0$;

\vrum

\item The set $\Gamma ^{(\omega _0)}_{0,s}(\rr d)$ consists of all $a\in C^\infty (\rr d)$
such that \eqref{Eq:GammaDef2} holds for every $h>0$.
\end{enumerate}
\end{defn}

\par

\begin{rem}\label{Remark:Gammaclasses}
We have
$$
\mascP \subseteq \mascP _{E,s_1}^0\subseteq \mascP _{E,s_1}
\subseteq \mascP _{E,s_2}^0, \qquad s_2<s_1.
$$
Hence, despite that $\Gamma ^{(\omega _0)}_{0,s}\subseteq
\Gamma ^{(\omega _0)}_s\subseteq S^{(\omega _0)}$ holds for every
$\omega _0$,  we have
\begin{align*}
\Gamma ^{(\omega )}_{0,s} &\nsubseteq
\bigcup _{\omega _0\in \mascP} S^{(\omega _0)}
\intertext{for some $\omega \in \mascP _{E,s}^0$, and}
\Gamma ^{(\omega )}_{0,s} &\nsubseteq
\bigcup _{\omega _0\in \mascP _{E,s}^0} \Gamma ^{(\omega _0)}_{s}
\end{align*}
for some $\omega \in \mascP _{E,s}$.
\end{rem}

\par

%%%%%%%%%%%%%%%%%%%%%%%%%%%%%%%%%%
\section{Continuity for pseudo-differential operators with symbols in
$\Gamma ^{(\omega )}_{s}$ and $\Gamma ^{(\omega )}_{0,s}$}\label{sec2}
%%%%%%%%%%%%%%%%%%%%%%%%%%%%%%%%%%

\par

In this section we discuss continuity for operators in $\op (\Gamma _{s}^{(\omega
_0)})$ and $\op (\Gamma _{0,s}^{(\omega _0)})$ when acting on a general class
of modulation spaces. In
Theorem \ref{p3.2} below it is proved that if $\omega ,\omega
_0\in\mascP _{E,s}^0$, $A\in \GL (d,\mathbf R)$ and $a\in \Gamma _{0,s}^{(\omega _0)}$, then
$\op _A(a)$ is continuous from $M(\omega _0\omega ,\mascB )$ to
$M(\omega _0,\mascB )$. This gives an analogy to \cite[Theorem 3.2]{To14} in the framework
of operator theory and Gelfand-Shilov classes.

\par

We need some preparations before discussing these mapping properties. The following
result shows that for any weight in $\mascP _E$, there are equivalent weights that 
satisfy strong Gevrey regularity.

\par

\begin{prop}\label{Prop:EquivWeights}
Let $\omega \in \mascP _E(\rr d)$ and $s >0$.
Then there exists a weight $\omega _0\in \mascP _E(\rr d)\cap C^\infty (\rr d)$
such that the following is true:
\begin{enumerate}
\item $\omega _0\asymp \omega $;

\vrum

\item $|\partial ^{\alpha}\omega _0(x)|\lesssim \omega _0(x) h^{|\alpha |}\alpha !^s
\asymp \omega (x) h^{|\alpha |}\alpha !^s$ for every $h>0$.
\end{enumerate}
\end{prop}

\par

\begin{proof}
We may assume that $s<\frac 12$. It suffices to prove that (2) should hold for some $h>0$.
Let $\phi _0\in \Sigma _{1-s}^s(\rr d)\setminus 0$, and let
$\phi =|\phi _0|^2$. Then $\phi \in \Sigma _{1-s}^s(\rr d)$ is non-negative.
In particular,
$$
|\partial ^\alpha \phi (x)| \lesssim h^{|\alpha |}\alpha !^s e^{-c|x|^{\frac 1{1-s}}},
$$
for every $h>0$ and $c>0$. We set $\omega _0=\omega *\phi$.

\par

Then
\begin{align*}
|\partial ^\alpha \omega _0(x)| &= \left | \int \omega (y)(\partial ^\alpha \phi )(x-y)
\, dy  \right |
\\[1ex]
&\lesssim 
h^{|\alpha |}\alpha !^s \int \omega (y)e^{-c|x-y|^{\frac 1{1-s}}} \, dy
\\[1ex]
&\lesssim 
h^{|\alpha |}\alpha !^s \int \omega (x+(y-x))e^{-c|x-y|^{\frac 1{1-s}}} \, dy
\\[1ex]
&\lesssim 
h^{|\alpha |}\alpha !^s \omega (x) \int e^{-\frac c2|x-y|^{\frac 1{1-s}}} \, dy
\asymp h^{|\alpha |}\alpha !^s \omega (x),
\end{align*}
where the last inequality follows \eqref{Eq:ModWeightProp} and the fact that
$\phi$ is bounded by a super exponential function. This gives the first part of (2).

\par

The equivalences in (1) follows in the same way as in e.{\,}g. \cite{To18}. More precisely,
by \eqref{Eq:ModWeightProp} we have
\begin{align*}
\omega _0(x) &= \int \omega (y)\phi (x-y)\, dy = \int \omega (x+(y-x))\phi (x-y)\, dy
\\[1ex]
&\lesssim \omega (x) \int e^{c|x-y|}\phi (x-y)\, dy \asymp \omega (x).
\end{align*}
In the same way, \eqref{Eq:ModWeightPropCons} gives
\begin{align*}
\omega _0(x) &= \int \omega (y)\phi (x-y)\, dy = \int \omega (x+(y-x))\phi (x-y)\, dy
\\[1ex]
&\gtrsim \omega (x) \int e^{-c|x-y|}\phi (x-y)\, dy \asymp \omega (x),
\end{align*}
and (1) as well as the second part of (2) follow.
\end{proof}

\par

The next result shows that $\Gamma ^{(\omega )}_{s}$ and
$\Gamma ^{(\omega )}_{0,s}$ can be characterised in terms of estimates of short-time
Fourier transforms. 

\par

\begin{prop}\label{Prop:CharGammaSTFT}
Let $s\ge 1$, $\phi \in \maclS _s(\rr d)\setminus 0$, and let
$f\in \maclS '_{1/2}(\rr d)$. Then the following is true:
\begin{enumerate}
\item If $\omega \in \mascP _{E,s}^0(\rr d)$, then $f\in C^\infty (\rr d)$ and satisfies
\begin{equation}\label{GelfRelCond1}
|\partial ^\alpha f(x)|\lesssim \omega (x)h^{|\alpha |}\alpha !^s,
\end{equation}
for some $h>0$, if and only if
\begin{equation}\label{stftcond1}
|V_\phi f(x,\xi )|\lesssim \omega (x)e^{-r|\xi |^{\frac 1s}},
\end{equation}
for some $r>0$;

\vrum

\item If $\omega \in \mascP _{E,s}(\rr d)$ and in addition $\phi \in \Sigma _s(\rr d)$,
then $f\in C^\infty (\rr d)$ and satisfies \eqref{GelfRelCond1} 
for every $h>0$ (resp. for some $h>0$), if and only if
\eqref{stftcond1} holds true for every $r>0$ (resp. for some $r>0$).
\end{enumerate}
\end{prop}

\par

%Later on our aim is to prove that the inclusions in Proposition \ref{prop1}
%are in fact equalities.

\par

\begin{proof}
We shall follow the proof of Proposition 3.1 in \cite{CaTo}. We only prove (2), and
then when \eqref{GelfRelCond1} or \eqref{stftcond1}
are true for every $\ep >0$. The other cases follow by similar arguments and
are left for the reader.

\par

Assume that $\phi \in \Sigma _s(\rr d)$, $\omega \in \mascP _{E,s}(\rr d)$
and that \eqref{GelfRelCond1} holds for every $\ep >0$. Then for every
$x \in \rr d$ the function
$$
y\mapsto F_x(y)\equiv f(y+x)\overline{\phi (y)}
$$
belongs to $\Sigma _s$, and $\omega (x+y)\lesssim e^{h_0|y|^{\frac 1s}}\omega (x)$ for some
$h_0>0$. By a straight-forward application of Leibnitz formula and the facts that
$$
|\partial ^\alpha \phi (x)|\lesssim \ep ^{|\alpha |}\alpha !^s e^{-h|x|^{\frac 1s}}
\quad \text{and}\quad
\omega (x+y)\lesssim \omega (x)e^{h_0|y|^{\frac 1s}}
$$
for some $h_0>0$ and every $\ep ,h>0$ we get
$$
|\partial _y^\alpha F_x(y)| \lesssim \omega (x)e^{-h|y|^{\frac 1s}}
\ep ^{|\alpha |}\alpha !^s,
$$
for every $\ep , h >0$. In particular,
\begin{equation}\label{FxEsts}
|F_x(y)|\lesssim \omega (x)e^{-h|y|^{\frac 1s}}\quad \text{and}\quad
|\widehat F_x(\xi )|\lesssim \omega (x)e^{-h|\xi |^{\frac 1s}},
\end{equation}
for every $h>0$. Since $|V_\phi f(x,\xi )| = |\widehat F_x(\xi )|$, the estimate
\eqref{stftcond1} follows from the second inequality in \eqref{FxEsts}. This
shows that if \eqref{GelfRelCond1} holds for every $\ep >0$, then \eqref{stftcond1}
holds for every $\ep >0$.

\par

Next suppose that \eqref{stftcond1} holds for every $\ep >0$. By
Fourier's inversion formula we get
$$
f(x) = (2\pi)^{-\frac d2} \nm \phi{L^2}^{-2} \iint _{\rr {2d}} V_\phi f(y,\eta )
\phi (x-y)e^{i\scal x\eta}\, dyd\eta .
$$
By differentiation and the fact that $\phi \in \Sigma _s$ we get
\begin{multline*}
|\partial ^\alpha f(x)| \asymp \left |
\sum _{\beta \le \alpha} {\alpha \choose \beta} i^{|\beta|}
\iint _{\rr {2d}} \eta ^\beta V_\phi f(y,\eta )  (\partial ^{\alpha -\beta }\phi )(x-y)
e^{i\scal x\eta}\, dyd\eta
\right |
\\[1ex]
\le
\sum _{\beta \le \alpha} {\alpha \choose \beta}
\iint _{\rr {2d}} |\eta ^\beta V_\phi f(y,\eta )  (\partial ^{\alpha -\beta }\phi )(x-y)|
\, dyd\eta
\\[1ex]
\lesssim
\sum _{\beta \le \alpha} {\alpha \choose \beta}
\iint _{\rr {2d}} |\eta ^\beta \omega (y)e^{-\ep _3|\eta |^{\frac 1s}} 
(\partial ^{\alpha -\beta }\phi )(x-y)|
\, dyd\eta
\\[1ex]
\lesssim
\sum _{\beta \le \alpha} {\alpha \choose \beta}
\ep _2 ^{|\alpha - \beta |} (\alpha -\beta )!^s
\iint _{\rr {2d}} |\eta ^\beta |\omega (y)e^{-\ep _3|\eta |^{\frac 1s}}
e^{-\ep _1|x-y|^{\frac 1s}} \, dyd\eta ,
\end{multline*}
for every $\ep _1,\ep _2,\ep _3>0$. Since
$$
|\eta ^\beta e^{-\ep _3|\eta |^{\frac 1s}}|\lesssim \ep _2^{|\beta|}\beta !^s
e^{-\ep _3|\eta |^{\frac 1s}/2},
$$
when $\ep _3$ is chosen large enough compared to $\ep _2^{-1}$, we get
\begin{multline*}
|\partial ^\alpha f(x)|
\\[1ex]
\lesssim \ep _2 ^{|\alpha |}
\sum _{\beta \le \alpha} {\alpha \choose \beta}
(\beta !
(\alpha -\beta )!)^s \iint _{\rr {2d}} \omega (y)e^{-\ep _3|\eta |^{\frac 1s}/2} 
e^{-\ep _1|x-y|^{\frac 1s}} \, dyd\eta
\\[1ex]
\lesssim (2\ep _2)^{|\alpha |}\alpha !^s \int _{\rr n} \omega (y)
e^{-\ep _1|x-y|^{\frac 1s}} \, dy
\end{multline*}

\par

Since $\omega (y)\le \omega (x)e^{h_0|x-y|^{\frac 1s}}$ for some $h_0\ge 0$
and $\ep _1$ can be chosen arbitrarily large, it follows from the last estimate that
$$
|\partial ^\alpha f(x)| \lesssim (2\ep _2)^{|\alpha |}\alpha !^s
\omega (x),
$$
for every $\ep _2>0$, and the result follows.
\end{proof}

\par

The following result is now a straight-forward consequence of the previous proposition
and the definitions.

\par

\begin{prop}\label{Prop:GammaModIdent}
Let $s\ge 1$, $q\in (0,\infty ]$, $\omega _0\in \mascP _{E,s}(\rr d)$ and let $\omega _r(x,\xi )
= \omega _0(x)e^{-r|\xi |^{\frac 1s}}$ when $x,\xi \in \rr d$. Then
$$
\bigcup _{r>0}M^{\infty ,q}_{(1/\omega _r)}(\rr d) = \Gamma ^{(\omega _0)}_{s}(\rr d)
\quad \text{and}\quad
\bigcap _{r>0}M^{\infty ,q}_{(1/\omega _r)}(\rr d) = \Gamma ^{(\omega _0)}_{0,s}(\rr d).
$$
\end{prop}

\par

The following lemma is a consequence of
Theorem 4.6 in \cite{CaTo}.
 
\par

\begin{lemma}\label{Somega}
Let $s\ge 1$ $\omega \in \mascP _{E}(\rr {2d})$, $A_1,A_2\in \GL (d,\mathbf R )$, and
that $a_1,a_2\in \Sigma _1'(\rr {2d})$ are such that
$\op _{A_1}(a_1)=\op _{A_2}(a_2)$. Then
\begin{alignat*}{3}
a_1 &\in \Gamma_{s}^{(\omega )}(\rr {2d})&\qquad &\Leftrightarrow &\qquad a_2
&\in \Gamma_{s}^{(\omega )}(\rr {2d}) 
\intertext{and}
a_1 &\in \Gamma_{0,s}^{(\omega )}(\rr {2d})&\qquad &\Leftrightarrow &\qquad a_2
&\in \Gamma_{0,s}^{(\omega )}(\rr {2d}) .
\end{alignat*}
\end{lemma}

\par

We have now the following result.

\par

\begin{thm}\label{p3.2}
Let $A\in \GL (d,\mathbf R)$, $s\ge 1$, $\omega ,\omega _0\in\mascP _{E,s}^0(\rr {2d})$,
$a\in \Gamma ^{(\omega _0)}_{s}(\rr {2d})$, and that $\mascB$
is an invariant BF-space on $\rr {2d}$. Then
$\op _A(a)$ is continuous from $M(\omega _0\omega ,\mascB )$
to $M(\omega ,\mascB )$.
\end{thm}

\par

We need some preparations for the proof, and start with the following lemma.
%by recalling Minkowski's
%inequality in a somewhat general form. Assume that
%$d\mu$ is a positive measure, and that $f\in L^1(d\mu ;\mascB )$ for some
%Banach space $\mascB$. Then Minkowski's inequality asserts that
%$$
%\Big \Vert \int f(x)\, d\mu (x)\Big \Vert _{\mascB} \le \int  \nm {f(x)}{\mascB}\, d\mu (x).
%$$
%
%\par
%
%We also need some lemmas.

\par

\begin{lemma}\label{Lemma:PrepReThm3.2A}
Suppose $s\ge 1$, $\omega \in \mascP _E(\rr {d_0})$ and that $f\in C^\infty
(\rr {d+d_0})$ satisfies
\begin{equation}\label{Eq:GevWeightEst}
|\partial ^\alpha f(x,y)|\lesssim h^{|\alpha |}\alpha !^se^{-r|x|^{\frac 1s}}\omega (y),
\quad \alpha \in \nn {d+d_0}
\end{equation}
for some $h>0$ and $r>0$. Then there are $f_0\in C^\infty (\rr {d+d_0})$
and $\psi \in \maclS _s(\rr d)$ such that \eqref{Eq:GevWeightEst} holds
with $f_0$ in place of $f$ for some $h>0$ and $r>0$, and
$f(x,y)= f_0(x,y)\psi (x)$.
\end{lemma}

\par

\begin{proof}
By Proposition \ref{Prop:EquivWeights}, there is a submultiplicative weight
$v_0\in \mascP _{E,s}(\rr d)\cap C^\infty (\rr d)$ such that
\begin{align}
v_0(x)&\asymp  e^{\frac r2|x |^{\frac 1s}}\label{Eq:v0AsympEst}
\intertext{and}
|\partial ^\alpha v_0(x)| &\lesssim h^{|\alpha |}\alpha !^sv_0(x),\qquad \alpha \in \nn d
\label{Eq:vDerEst}
\intertext{for some $h>0$. Since $s\ge 1$, a straight-forward application of Fa{\`a} di Bruno's formula
on \eqref{Eq:vDerEst} gives}
\left | \partial ^\alpha \left (\frac 1{v_0(x)}\right ) \right | &\lesssim
h^{|\alpha |}\alpha !^s\cdot \frac 1{v_0(x)},\qquad \alpha \in \nn d
\tag*{(\ref{Eq:vDerEst})$'$}
\end{align}
for some $h>0$. It follows from \eqref{Eq:v0AsympEst} and \eqref{Eq:vDerEst}$'$
that if $\psi =1/v$, then $\psi \in \maclS _s(\rr d)$. Furthermore, if $f_0(x,y)=f(x,y)v_0(x)$,
then an application of Leibnitz formula gives
\begin{multline*}
|\partial ^{\alpha}_x\partial ^{\alpha _0}_yf_0(x,y)|\lesssim
\sum _{\gamma \le \alpha }{{\alpha} \choose {\gamma}} |\partial ^\delta _x
\partial ^{\alpha _0}_yf(x,y)|
\, |\partial ^{\alpha -\delta }v_0(x)|
\\[1ex]
\lesssim
h^{|\alpha |+|\alpha _0|} \sum _{\gamma \le \alpha }{{\alpha} \choose {\gamma}}
(\gamma !\alpha _0!)^{s}e^{-r|x|^{\frac 1s}}\omega (y)(\alpha -\gamma )!^sv_0(x)
\\[1ex]
\lesssim
(2h)^{|\alpha |+|\alpha _0|}(\alpha !\alpha _0!)^se^{-r|x|^{\frac 1s}}v_0(x)\omega (y)
\\[1ex]
\asymp
(2h)^{|\alpha |+|\alpha _0|}(\alpha !\alpha _0!)^se^{-\frac 2r|x|^{\frac 1s}}v_0(x)\omega (y)
\end{multline*}
for some $h>0$, which gives the desired estimate on $f_0$, The result now
follows since it is evident that $f(x,y)= f_0(x,y)\psi (x)$.
\end{proof}

\par

For the next lemma we recall that for any $a\in \Sigma _1'(\rr {2d})$
there is a unique $b\in \Sigma _1'(\rr {2d})$ such that
$\op (a)^*=\op (b)$, and then $b(x,\xi )= e^{i\scal {D_\xi}{D_x}}\overline {a(x,\xi )}$
in view of \cite[Theorem 18.1.7]{Ho1}. Furthermore, by the latter equality and
\cite[Theorem 4.1]{CaTo} it follows that
$$
a\in \Gamma ^{(\omega )}_{s}(\rr {2d})
\quad \Leftrightarrow \quad
b\in \Gamma ^{(\omega )}_{s}(\rr {2d}).
$$

\par

\begin{lemma}\label{Lemma:PrepReThm3.2}
Let $s\ge 1$, $\omega \in \mascP _{E,s}^0(\rr {2d})$, $\vartheta \in
\mascP _{E,s}^0(\rr {d})$ and $v\in \mascP _{E,s}^0(\rr d)$ be such that
$v$ is submultiplicative, $\omega \in \Gamma ^{(\omega )}_{0,s}(\rr {2d})$
is $v\otimes v$-moderate, $\vartheta =v^{-\frac 12}$ and $\vartheta
\in \Gamma ^{(\vartheta )}_{0,s}(\rr d)$.
Also let $a\in \Gamma ^{(\omega )}_{s}(\rr {2d})$, choose $b\in
\Gamma ^{(\omega )}_{s}(\rr {2d})$ such that $\op (b)=\op (a)^*$,
$f\in \maclS _s(\rr d)$, $\phi \in \Sigma _s(\rr d)$, $\fy
= \widehat \phi \cdot v$,
\begin{align}
\Phi (x,\xi ,y,\eta ) &= \frac {b(y,\xi +\eta )}{\omega (x,\xi
)v(x-y)v(\eta )}\label{Phidef}
\intertext{and}
H(x,\xi ,y) &= v(x-y)\int \Phi (x,\xi ,y,\eta )\fy (\eta )
e^{i\scal {y-x}{\eta}}\, d\eta .\notag
\end{align}
Then
\begin{equation}\label{stftpseudoform}
V_\phi (\op (a)f)(x,\xi ) = (2\pi )^{-d} (f,e^{i\scal \cdo \xi
}H(x,\xi ,\cdo ))\omega (x,\xi
).
\end{equation}
Furthermore the following is true:
\begin{enumerate}
\item $H\in C^\infty (\rr {3d})$ and satisfies
\begin{equation}\label{Eq:DerHEst}
|\partial^\alpha H(x,\xi ,y)| \lesssim h_0^{|\alpha |}\alpha !^se^{-r_0|x-y|^{\frac 1s}},
\end{equation}
for some $h_0,r_0>0$;

\vrum

\item there are functions $H_0\in C^\infty (\rr {3d})$ and $\phi _0\in \maclS _s(\rr d)$
such that
\begin{equation}\label{Eq:HProd}
H(x,\xi ,y) = H_0(x,\xi ,y)\phi _0(y-x),
\end{equation}
and such that \eqref{Eq:DerHEst} holds for some $h_0,r_0>0$,
with $H_0$ in place of $H$.
\end{enumerate}
\end{lemma}

\par

\begin{proof}
When proving the first part, we shall use some ideas in the proof of
\cite[Lemma 3.3]{To14}. By straight-forward computations we get
\begin{multline}\label{Eq:storformel}
V_\phi (\op (a)f)(x,\xi ) = (2\pi )^{-\frac d2}(\op (a)f, 
\phi (\cdo -x)\, e^{i\scal \cdo \xi})
\\[1ex]
= (2\pi )^{-\frac d2}(f, \op (b)(\phi (\cdo -x)\, e^{i\scal \cdot \xi}))
\\[1ex]
= (2\pi )^{-d} (f,e^{i\scal \cdo \xi }H_1(x,\xi ,\cdo
))\omega (x,\xi ),
\end{multline}
where
\begin{multline*}
H_1(x,\xi ,y) = (2\pi )^{\frac d2}e^{-i\scal y\xi}(\op (b)(\phi (\cdo -x)
\, e^{i\scal \cdot \xi}))(y)/\omega (x,\xi )
\\[1ex]
= \int \frac {b(y,\eta )}{\omega (x,\xi )}\widehat \phi (\eta -\xi )
e^{-i\scal {x-y}{\eta -\xi}}\, d\eta 
\\[1ex]
= v(x-y) \int \Phi (x,\xi ,y,\eta -\xi )\fy (\eta -\xi )e^{-i\scal {x-y}{\eta -\xi}}\, d\eta .
\end{multline*}
If $\eta -\xi$ are taken as new variables of integrations,
it follows that the right-hand side is equal to $H(x,y,\xi )$. This
gives the first part of the lemma.

\par

In order to prove (1), let
$$
\Phi _0(x,\xi ,y,\eta ) = \Phi (x,\xi ,y,\eta )\fy (\eta ),
$$
and let $\Psi =\mascF _4\Phi _0$, where $\mascF _4\Phi _0$ is the partial
Fourier transform of $\Phi _0(x,\xi ,y,\eta )$ with respect to the $\eta$ variable.
Then it follows from the assumptions that
$$
|\partial ^\alpha \Phi _0(x,\xi ,y,\eta )|
\lesssim h_0^{|\alpha |}\alpha !^se^{-r_0|\eta |^{\frac 1s}},
$$
for some $h_0,r_0>0$, which shows that $\eta \mapsto \Phi _0(x,\xi ,y,\eta )$
is an element in $\maclS _s(\rr d)$ with values in $\Gamma ^{(1)}_{s}(\rr {3d})$.
As a consequence, $\Psi$ satisfies
$$
|\partial^\alpha \Psi (x,\xi ,y_1 ,y_2 )|
\lesssim h_0^{|\alpha |}\alpha !^se^{-r_0|y_2|^{\frac 1s}},
$$
for some $h_0,r_0>0$. The assertion (1) now follows from the latter
estimate, Leibnitz rule and the fact that
$$
H(x,\xi ,y) = v(x-y)\Psi (x,\xi ,x-y)
$$
%
%
%
%Hence
%$$
%|\partial^\alpha (\Psi (x,\xi ,\zeta ,\zeta )v(\zeta ))|
%\lesssim h_0^{|\alpha |}\alpha !^se^{-r_0|\zeta |^{\frac 1s}}
%$$
%for some $h_0,r_0>0$.
%
%\par
%
%By letting $H_2(x,\xi ,\cdo )$ be the partial Fourier transform of
%$\Psi (x,\xi ,\zeta ,\zeta )v(\zeta )$ with respect to the $\zeta$ variable,
%it follows that
%%%
%\begin{equation}\label{Eq:H2Est}
%|\partial^\alpha H_2(x,\xi ,y)|
%\lesssim h_0^{|\alpha |}\alpha !^se^{-r_0|y|^{\frac 1s}}
%\end{equation}
%%%
%for some $h_0,r_0>0$. The assertion (1) now follows from the latter estimate
%and the fact that $H(x,\xi ,y)= H_2(x,\xi ,x-y)$.

\par

In order to prove (2) we notice that \eqref{Eq:DerHEst} shows that
$y\mapsto H(x,\xi ,x-y)$ is an element in $\maclS _s(\rr d)$ with values
in $\Gamma ^{(1)}_{s}(\rr {2d})$. By Lemma \ref{Lemma:PrepReThm3.2A} there are
$H_2\in C^\infty (\rr {3d})$ and $\phi _0\in \maclS _s(\rr d)$ such that
\eqref{Eq:DerHEst} holds for some $h_0,r_0>0$ with $H_2$ in place of $H$,
and
$$
H(x,\xi ,x-y)= H_2(x,\xi ,x-y)\phi _0(-y).
$$
This is the same as (2), and the result follows.
\end{proof}

\par

\begin{proof}[Proof of Theorem \ref{p3.2}]
We may assume that $A=0$. Let $g=\op (a)f$. By
Lemma \ref{Lemma:PrepReThm3.2} we have
\begin{multline*}
V_\phi g(x,\xi )
=
(2\pi )^{-\frac d2}\mascF ((f\cdot \overline {\phi _0(\cdo -x)}) \cdot H_0(x,\xi ,\cdo ))(\xi )
\omega (x,\xi )
\\[1ex]
=
(2\pi )^{-d}\mascF ((f\cdot \overline {\phi _0(\cdo -x)})) *(\mascF  (H_0(x,\xi ,\cdo )))(\xi )
\omega (x,\xi )
\\[1ex]
=
(2\pi )^{-d} (V_{\phi _0}f)(x,\cdo ) * (\mascF  (H_0(x,\xi ,\cdo )))(\xi )
\omega (x,\xi )
\end{multline*}
Since $\omega$ and $\omega _0$ belongs to $\mascP _{E,s}^0(\rr {2d})$, (2) in Lemma
\ref{Lemma:PrepReThm3.2} gives
\begin{equation*}
|V_\phi g(x,\xi )\omega _0(x,\xi )| \lesssim |V_{\phi _0}f)(x,\cdo )\omega (x,\cdo )
\omega _0(x,\cdo )| * e^{-\frac {r_0}2|\cdo |^{\frac 1s}}.
\end{equation*}
Here we have used the fact that
$$
\omega (x,\xi )\omega _0(x,\xi )\lesssim \omega (x-y,\xi )\omega _0(x-y,\xi )
e^{\frac {r_0}2|\cdo |^{\frac 1s}}.
$$

\par

By applying the $\mascB$ norm we get for some $v\in \mascP _{E,s}^0(\rr d)$,
\begin{multline*}
\nm g{M(\omega _0,\mascB )}\lesssim \nm {|V_{\phi _0}f) \cdot \omega \cdot \omega _0|
* e^{-r_0|\cdo |^{\frac 1s}}\otimes \delta _0}{\mascB}
\\[1ex]
\le
\nm {V_{\phi _0}f) \cdot \omega \cdot \omega _0}{\mascB}\nm {e^{-r_0|\cdo |^{\frac 1s}}v}{L^1}
\asymp \nm f{M(\omega \cdot \omega _0,\mascB )}.
\end{multline*}
This gives the result.
\end{proof}

\par

By similar arguments as in the proof of Theorem \ref{p3.2} and Lemma
\ref{Lemma:PrepReThm3.2}  we get the following.
The details are left for the reader.

\par

\begin{thm}\label{p3.2B}
Let $A\in \GL (d,\mathbf R)$, $s\ge 1$, $\omega ,\omega _0\in\mascP _{E,s}(\rr {2d})$,
$a\in \Gamma ^{(\omega _0)}_{0,s}(\rr {2d})$, and that $\mascB$
is an invariant BF-space on $\rr {2d}$. Then
$\op _A(a)$ is continuous from $M(\omega _0\omega ,\mascB )$
to $M(\omega ,\mascB )$.
\end{thm}

\par

\begin{lemma}\label{Lemma:PrepReThm3.2B}
Let $s\ge 1$, $\omega \in \mascP _{E,s}(\rr {2d})$, $\vartheta \in
\mascP _{E,s}(\rr {d})$ and $v\in \mascP _{E,s}(\rr d)$ be such that
$v$ is submultiplicative, $\omega \in \Gamma ^{(\omega )}_{0,s}(\rr {2d})$
is $v\otimes v$-moderate, $\vartheta =v^{-\frac 12}$ and $\vartheta
\in \Gamma ^{(\vartheta )}_{0,s}(\rr d)$.
Also let $a\in \Gamma ^{(\omega )}_{0,s}(\rr {2d})$,
$f,\phi \in \Sigma _s(\rr d)$, $\phi _2= \phi v$, and let $\Phi$ and $H$ be as in Lemma
\ref{Lemma:PrepReThm3.2}.
Then \eqref{stftpseudoform} and the following hold true:
\begin{enumerate}
\item $H\in C^\infty (\rr {3d})$ and satisfies \eqref{Eq:DerHEst}
for every $h_0,r_0>0$;

\vrum

\item there are functions $H_0\in C^\infty (\rr {3d})$ and $\phi _0\in \Sigma _s(\rr d)$
such that \eqref{Eq:HProd} holds,
and such that \eqref{Eq:DerHEst} holds for every $h_0,r_0>0$,
with $H_0$ in place of $H$.
\end{enumerate}
\end{lemma}

\par

We finish the section by discussing continuity for pseudo-differential operators
with symbols in $\Gamma ^{(\omega _0)}_{s}$ or in $\Gamma ^{(\omega _0)}_{0,s}$
when acting on quasi-Banach modulation spaces. More precisely, by straight-forward
computations it follows that if $\omega ,\omega _0 \in \mascP _{E,s}(\rr {2d})$
($\omega ,\omega _0\in \mascP _{E,s}^0(\rr {2d})$),
then
$$
\frac {\omega (x,\xi  )}{\omega 
(y,\eta )\omega _0(y,\eta )} \lesssim \frac {e^{r(|\xi -\eta |^{\frac 1s}+|y-x|^{\frac 1s})}}{\omega _0( x,\eta )}.
$$
holds for some $r>0$ (for every $r>0$). Hence the following result is
a straight-forward consequence of Propositions \ref{Prop:OpCont} and
\ref{Prop:GammaModIdent}, and Lemma \ref{Somega}.
(Cf. Definition \ref{Def:MixedPhaseShiftLebSpaces} for the definition of
phase split basis.)

\par

\begin{thm}\label{Thm:OpCont3}
Let $A\in \GL (d,\mathbf R)$, $s\ge 1$, $\omega ,\omega _0 
\in \mascP _{E,s}^0(\rr {2d})$, \mbox{$\mabfp \in (0,\infty]^{2d}$},
$E$ be a phase split basis of $\rr {2d}$, and let
$a\in \Gamma ^{(\omega _0)}_{s}(\rr {2d})$. Then
$\op _A(a)$ is continuous from $M(L^{\mabfp ,E}(\rr {2d}),
\omega _0\omega )$ to $M(L^{\mabfp ,E}(\rr {2d}),\omega )$.

\par

The same holds true with $\mascP _{E,s}$ and $\Gamma ^{(\omega _0)}_{0,s}$, or with
$\mascP$ and $S^{(\omega _0)}$ in place of
$\mascP _{E,s}^0$ and $\Gamma ^{(\omega _0)}_{s}$, respectively, at each
occurence.
\end{thm}

\par

\begin{cor}\label{Cor:OpCont3}
Let $A\in \GL (d,\mathbf R)$, $s\ge 1$ and $\omega ,\omega _0 
\in \mascP _{E,s}^0(\rr {2d})$, $p,q \in (0,\infty]$, and let
$a\in \Gamma ^{(\omega _0)}_{s}(\rr {2d})$. Then
$\op _A(a)$ is continuous from $M^{p,q}_{(\omega _0\omega )}(\rr d)$ to
$M^{p,q}_{(\omega )}(\rr d)$.

\par

The same holds true with $\mascP _{E,s}$ and $\Gamma ^{(\omega _0)}_{0,s}$, or with
$\mascP$ and $S^{(\omega _0)}$ in place of
$\mascP _{E,s}^0$ and $\Gamma ^{(\omega _0)}_{s}$, respectively, at each
occurence.
\end{cor}

\par

%%%%%%%%%%%%%%%%%%%%%%%%%%%%%%
\section{Examples}\label{sec3}
%%%%%%%%%%%%%%%%%%%%%%%%%%%%%%

\par

In this section we list some examples and show how the continuity results
of the pseudo-differential operators in the previous section leads to continuity
on certain Sobolev spaces, weighted Lebesgue spaces and on
$\Gamma ^{(\omega )}_{0,s}$ spaces themselves.

\par

In the examples here we consider pseudo-differential operators with symbols in
$\Gamma ^{(\omega )}_{0,s}$ spaces. By some modifications, we may
also deduce similar continuity results for operators with symbols in
$\Gamma ^{(\omega )}_{s}$ spaces.

\par

\begin{example}\label{Example:SobSp}
Let $s\ge 1$, $r,r_0\in \mathbf R$, $A\in \GL (d,\mathbf R)$, and let $H_r^2(\rr d)$
be the Sobolev space of all $f\in \Sigma _1'(\rr d)$ such that
$\widehat f\in L^2_{loc}(\rr d)$ and
$$
\nm f{H_r^2}\equiv
\left (
\int _{\rr d}|\widehat f(\xi )e^{r|\xi |^{\frac 1s}}|^2\, d\xi 
\right )^{\frac 12}
$$
is finite. If $a\in C^\infty (\rr {2d})$ satisfies
\begin{equation}\label{Eq:SymbGammaSpec1}
|\partial ^\alpha a(x,\xi )|\lesssim h^{|\alpha |}\alpha !^se^{r_0|\xi |^{\frac 1s}}
\end{equation}
for every $h>0$, then $\op _A(a)$ is continuous from $H^2_r(\rr d)$ to
$H^2_{r-r_0}(\rr d)$.

\par

In fact, if $\omega _r(x,\xi )= e^{r|\xi |^{\frac 1s}}\in \mascP _{E,s}(\rr {2d})$,
then it follows that the condition \eqref{Eq:SymbGammaSpec1} holds true for every $h>0$
is the same as $a\in \Gamma ^{(\omega _{r_0})}_{0,s}(\rr {2d})$.
By a straight-forward applications of Fourier's
inversion formula we also have $H_r^2(\rr d)=M^{2,2}_{(\omega _r)}(\rr d)$
(cf. the proof of \cite[Proposition 11.3.1]{Gc2}). The assertion
now follows from these observations and letting $\mascB = L^2(\rr {2d})$
in Theorem \ref{p3.2B}.
\end{example}

\par

\begin{example}\label{Example:LebSp}
Let $s\ge 1$, $r,r_0\in \mathbf R$, $A\in \GL (d,\mathbf R)$, and let $L_r^2(\rr d)$
be the set $L^2_{r}(\rr d)$ which consists of all $f\in L^2_{loc}(\rr d)$ such that
$$
\nm f{L_r^2}\equiv
\left (
\int _{\rr d}|f(x)e^{r|x|^{\frac 1s}}|^2\, dx 
\right )^{\frac 12}
$$
is finite. If $a\in C^\infty (\rr {2d})$ satisfies
\begin{equation*} %\label{Eq:SymbGammaSpec1}
|\partial ^\alpha a(x,\xi )|\lesssim h^{|\alpha |}\alpha !^se^{r_0|x|^{\frac 1s}}
\end{equation*}
for every $h>0$, then $\op _A(a)$ is continuous from $L^2_r(\rr d)$ to
$L^2_{r-r_0}(\rr d)$.

\par

In fact, if $\omega _r(x,\xi )= e^{r|x|^{\frac 1s}}\in \mascP _{E,s}(\rr {2d})$,
then it follows that the conditions $a$ is the same as
$a\in \Gamma ^{(\omega _{r_0})}_{0,s}(\rr {2d})$,
and that $L_r^2(\rr d)=M^{2,2}_{(\omega _r)}(\rr d)$
(cf. the proof of \cite[Proposition 11.3.1]{Gc2}).
The assertion now follows from these observations and Theorem
\ref{p3.2B}.
\end{example}

\par

\begin{example}\label{Example:GammaSp}
Let $s\ge 1$, $\vartheta \in \mascP _{E,s}(\rr {2d})$, $\vartheta _0=\vartheta (\cdo ,0)
\in \mascP _{E,s}(\rr {d})$, $\omega _r$ and $\omega _0$
be the same as in Proposition \ref{Prop:GammaModIdent},
$A\in \GL (d,\mathbf R)$, and let $a\in \Gamma ^{(\vartheta )}_{0,s}(\rr {2d})$.
Then $\op _A(a)$ is continuous from $\Gamma ^{(\omega _0)}_{0,s}(\rr d)$ to
$\Gamma ^{(\omega _0\vartheta _0)}_{0,s}(\rr d)$.

\par

In fact, by Theorem \ref{p3.2B} it follows that
\begin{equation}\label{Eq:ContMinfty1}
\op _A(a)\, :\, M^{\infty ,1}_{(1/\omega _r)}(\rr d)
\to
M^{\infty ,1}_{(1/(\omega _r\vartheta ))}(\rr d)
\end{equation}
is continuous.
Since $M^{\infty ,1}_{(1/\omega _r)}(\rr d)$ is decreasing with respect to
$r$ and that
$$
\omega _{r-r_0}\vartheta _0 \lesssim \omega _r\vartheta \lesssim
\omega _{r+r_0}\vartheta _0,
$$
for some fixed $r_0\ge 0$, Proposition \ref{Prop:GammaModIdent} shows that
$$
\bigcap _{r>0} M^{\infty ,1}_{(1/(\omega _r))}(\rr d) =
\Gamma ^{(\omega _0)}_{0,s}(\rr d)
\quad \text{and}\quad
\bigcap _{r>0} M^{\infty ,1}_{(1/(\omega _r\vartheta ))}(\rr d) =
\Gamma ^{(\omega _0\vartheta _0)}_{0,s}(\rr d).
$$
The asserted continuity now follows from these intersections and
\eqref{Eq:ContMinfty1}.
\end{example}

\par

\end{document}